\documentclass[12pt,reqno,oneside]{amsart}
\usepackage{amsmath,amsthm,amsfonts,amssymb}
\usepackage[mathscr]{eucal}
\usepackage{bbm}
\usepackage{indentfirst}
\usepackage{url}

\theoremstyle{plain}
\newtheorem{teo}{Theorem}[section]

\newtheorem{lem}[teo]{Lemma}

\theoremstyle{definition}
\newtheorem{defn}[teo]{Definition}

\numberwithin{equation}{section}

\def\bb1{{\mathbbm{1}}}

\begin{document}
	\baselineskip=15pt
	\title[The coverage ratio of the frog model]{The coverage ratio of the frog model on complete graphs}
	\author{Gustavo~O.~de Carvalho}
	\author{F\'abio~P.~Machado}
	\address
	{Institute of Mathematics and Statistics
		\\ University of S\~ao Paulo \\ Rua do Mat\~ao 1010, CEP
		05508-090, S\~ao Paulo, SP, Brazil - gustavoodc@ime.usp.br and fmachado@ime.usp.br}
	\noindent
	\thanks{Research supported by Capes (88887.676435/2022-00), CNPq (303699/2018-3 and 132598/2020-5), FAPESP (17/10555-0)}
	\keywords{complete graph, coverage ratio, frog model, random walks system.}
	\subjclass[2010]{60K35, 05C81}
	
	\date{\today}

%
%
%
%
%

\begin{abstract}
The frog model is a system of interacting random walks. Initially, there is one particle at each vertex of a connected graph $\mathcal{G}$. All particles are inactive at time zero, except for the one which is placed at the root of $\mathcal{G}$, which is active. At each instant of time, each active particle may
die with probability $1-p$. Once an active particle survives, it jumps on one of its nearest vertices, chosen with uniform probability, performing a discrete time simple symmetric random walk (SRW) on
$\mathcal{G}$. Up to the time it dies, it activates all inactive particles it hits along its way. From the moment they are activated on, every such particle starts to walk, performing exactly the same dynamics, independent of everything else.
In this paper, we take $\mathcal{G}$ as the $n-$complete
graph ($\mathcal{K}_n$, a finite graph with each pair of vertices
linked by an edge). We study the limit in $n$ of the coverage ratio, that is, the proportion of visited vertices by some active particle up to the end of the process, after all
active particles have died. 
\end{abstract}

\maketitle


\section{Introduction}
\label{cap:introducao}
The frog model is a system of interacting random walks on a given rooted connected graph, finite or infinite. Initially, the graph contains some configuration (maybe one per vertex) of inactive particles on its vertices. The particles at the root of the graph start out active and perform independent simple nearest-neighbor discrete time random walks. Whenever a vertex with inactive particles is first visited, all the particles at the vertex wake up and begin their own independent random walks, waking up inactive particles as they visit them. The lifetime of a particle, in case it is active, can be finite or infinite. This model is often associated with the dynamics of the spread of a rumor or the spread of a virus in a connected population.
A formal definition of the frog model can be found in \cite{phase_transition}.

Part of the literature on the frog model is focused on the case where active particles have infinite lifetime and $\mathcal{G}$ is an infinite connected graph. In this setup, a goal is to study recurrence/transience, that is, conditions for the root of the graph to be visited infinitely often, as in \cite{recorrencia2,recorrencia3,recorrencia6}. 
Another approach on infinite graphs is to study the limit shape of the set
of visited vertices, also known as \textit{shape theorem}, as found in \cite{shape_theorem}.
On finite graphs and particles with infinite lifetime, one can study the first time where each vertex of the graph is visited
at least once (the \textit{coverage time}) as in \cite{vida_fixa2,cover_time}. Finite lifetime on finite graphs are considered in
\cite{vida_fixa,vida_fixa2}.


In this paper, we are interested in studying the limit coverage ratio for the frog model on $\mathcal{K}_n$, the $n-$complete graph, when the lifetime of an active particle follows a geometric random variable with parameter $1-p$. That
is a consequence of the fact that at each instant of time, indenpendently of everything else, each active particle may
die with probability $1-p$ or survive with probability $p$. 
We know from \cite{grafo_completo1} that, starting from a one particle per vertex configuration on $\mathcal{K}_n$, there exists a phase transition property, which means the existence of a non-trivial critical parameter $p_c$ such that the coverage ratio converges to zero in distribution for $p<p_c$, while this convergence does not occur for $p>p_c$. See \cite{grafo_completo2} for some interesting result with this setup. It is also shown in \cite{grafo_completo1} that $p_c=1/2$. Our main theorem extend the previous result, giving more details about the number of visited vertices, and proving also, for $p>p_c=1/2$, the exact limit in $n$ of the probability that the coverage ratio is strictly larger than zero.


Let us start with a basic definition

\begin{defn} 
	Let  $V_\infty = V_\infty(\mathcal{K}_n)$ be the number of vertices of $\mathcal{K}_n$, visited at least once, up to the time there are no more active particles. 
\end{defn}

Our main result is the following.

\begin{teo} \label{teorema:principal} 
	
Consider the frog model on $\mathcal{K}_n$, starting from the configuration with one particle per vertex. Then

(i) For $p\leq 1/2$, for every function $f:\mathbb{N}\rightarrow \mathbb{R}$ such that $\lim_{n \to \infty}f(n)=+\infty$,
\[\lim_{n \to \infty}P(V_\infty \leq f(n))=1.
\]

(ii) For $p>1/2$, there are constants $c>0$ e $c'\in (0,1)$ such that
\[\lim_{n \to \infty}P(V_\infty \leq c \log n)=\frac{1-p}{p},\]
\[\lim_{n \to \infty}P(V_\infty \geq c' n)=\frac{2p-1}{p}.
\]
\end{teo}

Observe that the function $f(n)$ in part \textit{(i)} of Theorem~\ref{teorema:principal} can be replaced by $c \log n$, for any $c>0$.
Still, part $(i)$ of Theorem~\ref{teorema:principal} is valid for any sequence of connected graphs. The proof we provide in this paper works also in this case. 

The ideas and techniques used in the proof of the Theorem~\ref{teorema:principal} are inspired by the study of the existence of a giant component on Erdös–Rényi graph (see \cite[sec. 11.2]{componente_gigante} for details). Some basic concepts of branching processes (see~\cite[sec. 5.4]{branching}) are also considered.

The intuition behind Theorem~\ref{teorema:principal} comes from the fact that if one considers a stage when there is a small number of visited vertices on a large graph, what one sees is that the next surviving particle will, most likely, visit an unvisited vertex, activating a new particle. Therefore, in the beginning of the process, the number of vertices visited is near to the number of individuals in a branching process that generates 2 offspring with probability $p$ and 0 offspring with probability $1-p$. In this setup, the extinction probability equals to $1$ for $p \leq 1/2$ and to $\frac{1-p}{p}$ for $p>1/2$. Besides that, for $p>1/2$, the branching process mentioned above grows infinitely with probability $\frac{2p-1}{p}$, an event that is in some sense analogous to $V_\infty$ being on the order of $n$.

%
%
%
%

The paper is organized as follows. 
In section~\ref{sec:proc_auxiliar} we describe an auxiliary process aiming to simplify the computations needed. In section~\ref{sec:prova_teorema} one finds the proofs of Theorem~$\ref{teorema:principal}$. The sub-critical phase ($p \leq 1/2$) is proved in subsection~\ref{sec:subcritica}, for any sequence of connected graphs, while the supercritical phase ($p > 1/2$) is proved in subsection~\ref{sec:supercritica}.

\section{The auxiliary process}
\label{sec:proc_auxiliar}
Let us define an auxiliary process starting from a small modification on the frog model. Let us start from the $n+1$-complete graph, $\mathcal{K}_{n+1}$, picking one of its vertices and defining it as its root. We denote the set of vertices of $\mathcal{K}_{n+1}$ by $\mathcal{V}(\mathcal{K}_{n+1})$.
At time zero, there is one active particle at the root and one inactive particle at each other vertices of $\mathcal{K}_{n+1}$. All particles present at time zero are declared \textit{original}. This is done because at some point, \textit{extra} particles are considered.

The auxiliary process has the following characteristics

\begin{itemize}
    \item It is taken in rounds, so that only one particle acts (dying or moving) in each round. The particle which is being considered each time, survives with probability $p$ and, if so, it jumps on one of its nearest vertices,
    chosen with uniform probability.
    \item Let $R$ be the time that the last original active particle dies. It is the time that the original frog model dies out. At this very time, every remaining inactive particles become \textit{extra} inactive particles. Besides that, a brand new \textit{extra} active particle is placed at the root of $\mathcal{K}_{n+1}$, so it acts in round $R+1$.
    \item From now on every time the last extra active particle dies, another extra active particle is placed at the root of $\mathcal{K}_{n+1}$.
\end{itemize}

Observe that there are only extra particles from the round $R+1$ on. Observe also that $R<\infty$ with probability 1, as the number of original particles activated up to a given time is at most $|\mathcal{V}(\mathcal{K}_{n+1})|=n+1<\infty$ and the lifetime of any active particle is given by a geometric random variable of parameter $1-p$. The lifetime is finite with probability 1, since we are restricted to $p < 1$.

Even though $R<\infty$, the auxiliary process as a whole is infinite, as it always renews itself when we consider an extra particle at the origin. This rule makes it possible to analyze the behavior of the particle (original or extra) that performs the $k$th round for any $k \in \mathbb{N}$, without the need to check whether the original process has finished or not.

Now, for $k\in \mathbb{N}$ we define the random variable $X_k$ which depends on what happens with the particle that acts in round $k$ during the auxiliary process:

\begin{itemize}
    \item $X_k=0$ if it dies.
    \item $X_k=1$ if it survives and jumps to a vertex which has been visited before round $k$.
    \item $X_k=2$ if it survives and jumps to a vertex which has never been visited before round $k$.
\end{itemize}

Therefore, $X_k$ can be interpreted as the number of descendants (in the branching process sense) of the particle that acts in the $k$th round. Observe that the sequence $X_1,X_2,...$ is neither independent nor identically distributed.

From the number of descendants in each round, we define the number of potentially active particles at the end of the $k$-th round as

\begin{equation} \label{def:pot_ativas}
A'_0:=1; \hspace{5mm} A'_k:=1+\sum_{j=1}^k (X_j-1), \hspace{1mm} k\geq 1.\end{equation}

The term potential comes from the fact that the previous expression also considers descendants of extra particles. We can disregard this by taking

\begin{equation}\label{def:r}
R=\inf\{k:A'_k=0\}
\end{equation}

\noindent
and setting the number of active particles at the end of the $k$th round as

\begin{equation} \label{def:ativas}
A_k:=A'_k\bb1_{(k<R)}, \hspace{2mm} k \in \mathbb{N}. 
\end{equation}

Similarly, we define the number of potentially visited vertices by the end of the $k$th round as

\begin{equation}\label{def:pot_visitados}
    V'_0:=1; \hspace{5mm} V'_k:=1+\sum_{j=1} ^{k} \bb1_{(X_j=2)}, \hspace{1mm} k\in \mathbb{N} 
\end{equation}

\noindent
and the number of vertices visited at the end of the $k$th round as

\begin{equation} \label{def:visitados}
V_k:=V'_k \bb1_{(k < R)} + V'_R \bb1_{(k \geq R)},\hspace{1mm} k\in \mathbb{N}.
\end{equation}

We define the total number of vertices visited during the whole process by

\begin{equation} \label{def:visitados_total}
V_\infty:=\lim_{k \to \infty} V_k=V_{R}.
\end{equation}

Note that $V_\infty=V_R$ is not affected by the extra particles, as they are triggered only after $R$. The extra particles are useful for computational purposes, as they make it possible to keep track of the number of potentially active particles and potentially visited vertices without having to worry about whether all the original particles have died or not.

It is important to note that $V_\infty$ also describes the total number of vertices visited in the original frog model. Therefore, the study of the frog model will be done from the auxiliary process, using $V_\infty$.

Note that the auxiliary process and its definitions could be done for any conected graph. In some cases one could have $R=\infty$ with positive probability. The auxiliary process fits well the complete graph due to the fact that some computations can be easily done, as

\begin{equation}\label{eq:dist_x}
P(X_k=x|V'_{k-1}=v)=\begin{cases}
    1-p & \text{if $x=0$,}\\
    \frac{p(v-1)}{n} & \text{if $x=1$,}\\
    \frac{p(n-v+1)}{n} & \text{if $x=2$.}
\end{cases}
\end{equation}

\section{Proofs} 
\label{sec:prova_teorema}


We start off with a basic probability result involving Chernoff bounds that will be useful later on.

\begin{lem}[\cite{chernoff}, Theorem 4.5] \label{lema:chernoff}
For $X$ a random variable with binomial distribution, $X \sim B(n,p)$, it holds that
\[P(X \leq E(X) -t) \leq \exp(\frac{-t^2}{2E(X)}), \hspace{3mm} t>0.\]
\end{lem}

\subsection{The sub-critical phase: $p\leq 1/2$} \label{sec:subcritica}

\begin{proof}[Proof of Theorem~\ref{teorema:principal} part (i)]

Consider $\{\mathcal{G}_n\}_{n\in \mathbb{N}}$ a sequence of connected graphs. Let us define a process such that every time there is an active particle jumping to a neighboring vertex, it chooses one that has never been visited before. Our aim is to define a process whose number of visited vertices at any given time dominates the same quantity computed for the auxiliary process defined in~Section~\ref{sec:proc_auxiliar}.

Let us define a sequence of random variables $\{X_j^+\}_{j\in \mathbb{N}}$

\begin{equation}\label{eq:compara_y}
X_j^+=\begin{cases}
    0 & \text{if $X_j=0$,}\\
    2 & \text{if $(X_j=1) \cup (X_j=2)$.}
\end{cases}
\end{equation}

Note that $X^+_1,X^+_2,...$ are independent and identically distributed, as they rely only on the survival event of particles. Consider also $X^+$ a random variable with the same distribution as $X^+_1$. It holds that

\begin{equation} \label{eq:esp_y}
    E(X^+)=2P(X^+=2)=2p\leq 1.
\end{equation}

Next, we consider a branching process where the number of children of each individual is distributed as $X^+$. As done for the auxiliary process, we keep track of this branching process by the number of descendants of each individual per turn.

Analogously to definitions (\ref{def:pot_ativas}), (\ref{def:r}), (\ref{def:ativas}), (\ref{def:pot_visitados}), (\ref{def:visitados})  and (\ref{def:visitados_total}), denote the number of potentially alive individuals at the end of the $k$th round of the branching process by

\begin{equation} \label{def:pot_ativas_y}
    A'^{+}_k:=1+\sum_{j=1}^k (X_j^+-1),
\end{equation}

\noindent
and the number of rounds until the extinction of this population by
\begin{equation} \label{def:r_y}
    R^{+}:=\inf\{k \in \mathbb{N}:A'^{+}_k=0\}.
\end{equation}

The total number of individuals in this branching process can be written as

\[V_\infty^{+}:=1+\sum_{j=1}^{R^{+}}\bb1_{(X_j^+=2)}.
\]

Note that (\ref{eq:esp_y}) tells us that the branching process with offspring distribution $X^+$ becomes extinct, and therefore its total number of individuals is finite, with probability 1. The extinction probability does not change if we count the descendants of one individual at a time each round. As, by hypothesis, $\lim_{n \to \infty}f(n)=+\infty $, it follows that:

\[\lim_{n \to \infty}P(V_\infty^+\leq f(n))=1.
\]

From ($\ref{eq:compara_y}$), we have that for every $k\in \mathbb{N}$ it holds that $X^+_k\geq X_k$, which by its turn implies that $A'^+_k\geq A'_k$ and so $R^+\geq R$. It is also true that $V_\infty^+\geq V_\infty$. So, we conclude that

\[\lim_{n \to \infty}P(V_\infty\leq f(n))=1.
\]
\end{proof}

\subsection{The super-critical phase: $p > 1/2$} \label{sec:supercritica}

\begin{proof}[Proof of Theorem~\ref{teorema:principal} part (ii)]
	As the case $p=1$ is trivial, we focus on $p \in (1/2,1)$, splitting the proof in the three following statements

\begin{itemize}
    \item[(a)] For $k_-=2 \frac{4p}{(1+2p)[\frac{2p-1}{8p+4}]^2}\log n$ and $k_+=(1-\frac{2}{1+2p})n$, it holds that $${\lim_{n \to \infty}P(R \in [k_-,k_+])=0}.$$
    \item[(b)] $\lim_{n \to \infty}P(V_\infty\leq c\log n|R < k_-)=1$ and ${\lim_{n \to \infty}P(V_\infty\geq c'n|R > k_+)=1}$.
    \item[(c)] $\lim_{n \to \infty}P(R < k_-)= \frac{1-p}{p}$.
\end{itemize}

Note that with these 3 items, the proof comes to an end. Next, we prove (a), (b) and (c).

\noindent
\textit{Proof of (a).} First we define 

\[k_-=k_-(n):=2 \frac{4p}{(1+2p)[\frac{2p-1}{8p+4}]^2}\log n;\]
\noindent
and
\[k_+=k_+(n):=(1-\frac{2}{1+2p})n.\]

See that $0< 1-\frac{2}{1+2p} < 1$ when $p> 1/2$.
For $k \in \mathbb{N}\cap[k_-,k_+]$, we may consider independent random variables $Y_1,...,Y_k$ such that for $j\in \{1,...,k \}$
\begin{equation}
\label{eq:dist_y}
P(Y_j=x)=
\begin{cases}
    1-p & \text{if $x=0$,}\\
    \frac{pk_+}{n} & \text{if $x=1$,}\\
    \frac{p(n-k_+)}{n} & \text{if $x=2$.}
\end{cases}
\end{equation}

For any $j\in \{1,...,k \}$ it is true that $V'_{j-1}-1\leq j \leq k \leq k_+$. Note that $Y_j$ can be interpreted as the number of descendants of the particle participating in the $k$th round, when we maximize the number of visited vertices (compare with the equations (\ref{eq:dist_y}) and (\ref{eq:dist_x})). So, by coupling, we can assume that $\sum_{j=1}^k X_j$ is always greater than $\sum_{j=1}^k Y_j$. Let us denote this fact by
$\sum_{j=1}^k X_j \succeq  \sum_{j=1}^k Y_j.$

Next, let us define define $a_k:=(\frac{2p}{1+2p}-\frac{1}{2})k$. Note that $0<\frac{2p}{1+2p}-\frac{1}{2}<1$ for $p>1/2$.
By using the Chernoff bounds from Lemma~\ref{lema:chernoff}, together with the fact that $$\sum_{j=1}^k \bb1_{(Y_j=2)} \sim Bin(k,\frac{p(n-k_+)}{n})$$ and (\ref{def:pot_ativas}), we have that

\begin{align} \label{eq:pot_ativas}
P(&A'_k\leq a_k+1)\nonumber\\
&=P(\sum_{j=1}^k X_j \leq k+a_k)\nonumber\\
&\leq P(\sum_{j=1}^k Y_j \leq k+a_k)\nonumber\\
&\leq P(\sum_{j=1}^k \bb1_{(Y_j=2)}\leq \frac{k+a_k}{2})\nonumber\\
&=P(\sum_{j=1}^k \bb1_{(Y_j=2)} \leq kp-\frac{pkk_+}{n}-k(p-\frac{1}{2})+\frac{pkk_+}{n}+\frac{a_k}{2})\\
&=P(\sum_{j=1}^k \bb1_{(Y_j=2)} \leq E(\sum_{j=1}^k \bb1_{(Y_j=2)})-k[\frac{2p-1}{8p+4}])\nonumber\\
&\leq \exp(-\frac{k^2[\frac{2p-1}{8p+4}]^2}{2kp(1-\frac{k_+}{n})})\nonumber\\
& \stackrel{(*)}{\leq}\exp(-\frac{k_-[\frac{2p-1}{8p+4}]^2(1+2p)}{4p})\nonumber\\
&=\exp(-2 \log n)=o(n^{-1}),\nonumber
\end{align}
where for two sequences of functions $f_1,f_2,...$ and $g_1,g_2,...$, we write $f_n=o(g_n)$ if $\lim_{n \to \infty}\frac{f_n}{g_n}=0$.
The Chernoff bounds can be applied as $\frac{2p-1}{8p+4}>0$ for $p>1/2$.
Observe that we can exchange $k$ by $k_-$ in (*) as $k_-\leq k$ and $\exp(-x)$ is decreasing in $x$.

Now, consider the set 
\[A:=\bigcup_{k\in \mathbb{N}\cap[k_-,k_+]} (A'_k\leq a_k+1 )\] 

In words, it means that for some ${k\in \mathbb{N}\cup[k_-,k_+]}$ there is at most $a_k+1$ potentially active particles. By the subaditivity of the probability measure and (\ref{eq:pot_ativas}), we have

\[P(A)\leq \sum_{k\in \mathbb{N}\cap[k_-,k_+]} P(A'_k\leq a_k+1)\leq k_+ o(n^{-1})=o(1).\]
Then 
$$P(A^c)=P(\cap_{k\in \mathbb{N}\cap [k_-,k_+]}(A'_k>a_k+1)) \stackrel{n\to \infty}{\rightarrow} 1.$$


As $R:=\inf\{k:A'_k=0 \},$
$$\lim_{n \to \infty}P(R \in [k_-,k_+])=0.$$

\noindent
\textit{Proof of (b).} From the definition of $R,$ we know that $X_j=0$ for some $j\in \{ 1,2,..,R\}$. Therefore, for $R<k_-$, from (\ref{def:pot_visitados}), (\ref{def:visitados}) and (\ref{def:visitados_total}), we know that 

$$V_\infty=V_R=V'_R=1+\sum_{j=1}^R \bb1_{(X_j=2)}\leq 1+ R -1<k_-. $$

So, we conclude that
\[\lim_{n \to \infty}P(V_\infty\leq k_-|R < k_-)=1.\]

\vspace{3mm}

Furthermore, for $\lfloor k_+ \rfloor=\max\{x\in \mathbb{Z}:x \leq k_+\}$, we have that the probability of having $\frac{k_++a_{k_+}}{2}=n(\frac{12p^2-4p-1}{4(1+2p)^2})$ or fewer potentially visited vertices equals to
\begin{equation}
\begin{aligned} \label{eq:vert_k+}
P(V'_{\lfloor k_+ \rfloor}\leq \frac{k_++a_{k_+}}{2})
&=P(\sum_{k=1}^{\lfloor k_+ \rfloor} \bb1_{(X_k=2)}+1 \leq \frac{k_++a_{k_+}}{2})\\
&=P(\sum_{k=1}^{\lfloor k_+ \rfloor} \bb1_{(X_k=2)} \leq \frac{(k_+-1)+(a_{k_+}-1)}{2})\\
&\stackrel{(*)}{\leq} P(\sum_{k=1}^{\lfloor k_+ \rfloor} \bb1_{(X_k=2)} \leq \frac{\lfloor k_+ \rfloor +a_{\lfloor k_+ \rfloor}}{2} )\\
&\leq P(\sum_{k=1}^{\lfloor k_+ \rfloor} \bb1_{(Y_k=2)} \leq \frac{\lfloor k_+ \rfloor +a_{\lfloor k_+ \rfloor}}{2} ) =o(1).
\end{aligned}
\end{equation}

The inequality  (*) holds as $\lfloor k_+ \rfloor \geq k_+-1$ and that, $a_k=ck$, for some constant $c\in(0,1)$. Therefore $a_{\lfloor k_+ \rfloor}=c \lfloor k_+ \rfloor \geq c(k_+ -1)=a_{k_+}-c \geq a_{k_+}-1$. 

Moreover, see that (\ref{eq:vert_k+}) is $o(1)$ because $\lfloor k_+ \rfloor \in \mathbb{N}\cap[k_-,k_+]$ and that term shows up in (\ref{eq:pot_ativas}).


Therefore
\[P(V'_{\lfloor k_+ \rfloor}\leq \frac{k_+ +a_{k_+}}{2}|R > k_+)P(R > k_+)+P(V'_{\lfloor k_+ \rfloor}\leq \frac{k_+ +a_{k_+}}{2}|R \leq k_+)P(R \leq k_+)\stackrel{n \to \infty}{\rightarrow} 0.\]

So, if $\lim_{n \to \infty} P(R > k_+)>0$ (see part c), it holds that $\lim_{n \to \infty}P(V'_{\lfloor k_+ \rfloor}>\frac{k_+ +a_{k_+}}{2}|R > k_+)=1$. In addition, by (\ref{def:pot_visitados}), (\ref{def:visitados}) and (\ref{def:visitados_total}) when $ R > k_+ \geq \lfloor k_+ \rfloor$, it is true that $V_\infty=V_R\geq V_{\lfloor k_+ \rfloor}=V'_{\lfloor k_+ \rfloor}$.

From this, we conclude that
\[\lim_{n \to \infty}P(V_\infty\geq \frac{k_+ +a_{k_+}}{2}|R > k_+)=1.\]

Remember that $k_+=c_1 n$ for a constant $c_1\in(0,1)$ and that $a_{k_+}=c_2 k_+=c_1 c_2 n$ for a constant $c_2\in(0,1)$, then $\frac{k_+ +a_{k_+}}{2}=c' n$ for some $c'\in (0,1)$, thus concluding the proof of (b).

\vspace{3mm}

\noindent
\textit{Proof of (c).}

Let us return to the random variable $X^+$ presented in sub-section~\ref{sec:subcritica}.

\begin{equation}\label{eq:dist_x+}
P(X^+=x)=
\begin{cases}
    1-p, & \text{if $x=0$,}\\
    p, & \text{if $x=2$.}\\
\end{cases}    
\end{equation}

Again, we consider a branching process with offspring distribution given by $X^+$, where the number of descendants of the individuals is evaluated in turns, and return to definitions (\ref{def:pot_ativas_y}), (\ref{def:r_y}), rewritten below.
\[ 
    A'^{+}_k:=1+\sum_{j=1}^k (X_j^+-1)
\]
representing the number of potentially active particles of this branching process at the end of the $k$th round and
\[
    R^{+}:=\inf\{k \in \mathbb{N}:A'^{+}_k=0\}
\]
representing the instant where the branching process ends. It is now necessary to define $R^+:=+\infty$ if $\{k \in \mathbb{N}:A'^{+}_k=0 \}=\emptyset$, since it can happen, with positive probability in the case $p>1/2$, of having a never ending process. Furthermore, we define
\begin{equation}\label{def:ativas+}
    A^{+}_k:=A'^+_k\mathbb{I}(k<R^+)
\end{equation}
representing the number of effectively active particles of this branching process at the end of the $k$th round.

Let us define the event for which the branching process with offspring distribution $X^+$ extinguishes, but after the round $k_-$ by
\[B^+_{k_-}:=\{k_-<R^+< +\infty\}.\] 

We define one last branching process, with the same rounds scheme, with offspring distribution $X^-$, where

\begin{equation}\label{eq:dist_x-}
P(X^-=x)=
\begin{cases}
    1-p & \text{if $x=0$,}\\
    \frac{pk_-}{n} & \text{if $x=1$,}\\
    \frac{p(n-k_-)}{n} & \text{if $x=2$.}
\end{cases}\end{equation}

Similarly, consider $X^-_1,X^-_2,...$ as identically distributed independent random variables with the same distribution of $X^-$ and define \[A'^{-}_k:=1 +\sum_{j=1}^k (X_j^- -1)\] as the number of potentially alive individuals at the end of the $k$th round. 

Furthermore, define

\[
    R^{-}:=\inf\{k \in \mathbb{N}:A'^{-}_k=0\}
\]
as the moment when the branching process with offspring given by $X^-$ ends. Consider that $R^-:=+\infty$ if $\{k \in \mathbb{N}:A'^{-}_k=0 \}=\emptyset$.


We have that $k \leq k_-\Rightarrow V'_{k-1}=1+\sum_{j=1}^{k-1} \bb1_{(X_j=2)} \leq k_-$. Therefore, comparing (\ref{eq:dist_x+}) and (\ref{eq:dist_x-}) with (\ref{eq:dist_x}), we conclude that $\sum_{j=1}^k X^-_j \preceq \sum_{j=1}^k X_j \preceq \sum_{j=1}^k X^+_j$, and so $A'^{-}_k \preceq A'_k\preceq A'^{+}_k$ for all $k\in \mathbb{N}\cap [1,k_-]$. We then see that it is possible to couple the three process so that $R^{+}\leq k_- \Rightarrow R\leq k_-$ and also $R\leq k_- \Rightarrow R^-\leq k_-$. 

So, it is true that

\begin{equation} \label{eq:confronto}
P(R^+<\infty) -P(B^+_{k_-})=P(R^{+}\leq k_-)\leq P(R \leq k_-) \leq P(R^-\leq k_-) \leq P(R^-<\infty).
\end{equation}

By properties of branching processes, $P(R^+<\infty)$ is the smallest positive solution of

\[P(R^+<\infty)=1-p+P(R^+<\infty)^2p.
\]

Solving the quadratic equation above, we realize that \[P(R^+<\infty)=\frac{1-p}{p}.\]

It also holds that $P(R^-<\infty)$ is the smallest positive solution of
\[P(R^-<\infty)=1-p+P(R^-<\infty)\frac{pk_-}{n}+P(R^-<\infty)^2\frac{p(n-k_-)}{n}.
\]

Solving the quadratic equation above, we see that \[P(R^-<\infty)=\frac{n(1-p)}{pn-pk_-}\stackrel{n \to \infty}{\rightarrow} \frac{1-p}{p}.\]

Denote $\lceil k_- \rceil=\min\{x\in \mathbb{Z}:x\geq k_-\}$. From (\ref{eq:pot_ativas}), it holds that $\lim_{n \to \infty}P(A'_{ \lceil k_- \rceil}>a_{\lceil k_- \rceil}+1)=1$ and therefore $\lim_{n \to \infty}P(A'^{+}_{ \lceil k_- \rceil}>a_{\lceil k_- \rceil}+1)=1$. Then, if $\lim_{n \to \infty} P(R^{+}> k_-)>0$ (which is true because $E(X^+)>1$), we have that $\lim_{n \to \infty}P(A'^{+}_{\lceil k_- \rceil}>a_{\lceil k_- \rceil}+1|R^{+}> k_-)=1$. Besides that, when $R^{+}> k_-$, it holds that $R^{+}\geq {\lceil k_- \rceil}$ and from (\ref{def:ativas+}), we have that
$A^{+}_{\lceil k_- \rceil}=A'^{+}_{\lceil k_- \rceil}$. Then, we conclude

\[\begin{aligned}
P(B_{k_-}^+)&= P(R^+>k_-,R^+<\infty)\\
&=P(R^+>k_-,R^+<\infty,A^{+}_{\lceil k_- \rceil}>a_{\lceil k_- \rceil}+1)+P(R^+>k_-,R^+<\infty,A^{+}_{\lceil k_- \rceil}\leq a_{\lceil k_- \rceil}+1) \\
&\leq \frac{P(R^+>k_-,R^+<\infty,A^{+}_{\lceil k_- \rceil}>a_{\lceil k_- \rceil}+1)}{P(R^+>k_-,A^{+}_{\lceil k_- \rceil}>a_{\lceil k_- \rceil}+1)}+\frac{P(R^+>k_-,A^{+}_{\lceil k_- \rceil}\leq a_{\lceil k_- \rceil}+1)}{P(R^+>k_-)}\\
& = P(R^+<\infty,|R^+>k_-,A^{+}_{\lceil k_- \rceil}>a_{\lceil k_- \rceil}+1)+P(A^{+}_{\lceil k_- \rceil}\leq a_{\lceil k_- \rceil}+1|R^+>k_-) \\
&\leq (\frac{1-p}{p})^{a_{\lceil k_- \rceil}+1}+P(A^{+}_{\lceil k_- \rceil}\leq a_{\lceil k_- \rceil}+1|R^+>k_-) \stackrel{n \to \infty}{\rightarrow}0
\end{aligned}
\]

where the last inequality is given by the extinction probability of at least $a_{\lceil k_- \rceil}+1$ independent branching processes, each with extinction probability $\frac{1-p}{p}$.


Applying to (\ref{eq:confronto}) all limits found, we conclude that

\[\lim_{n \to \infty} P(R<k_-)=\frac{1-p}{p}.\]
\end{proof}

\bibliographystyle{alpha}
\bibliography{sample}

\end{document}